\documentclass{amsart}

\usepackage{algorithm2e}
\RestyleAlgo{ruled}
\usepackage{amsfonts}
\usepackage{amsrefs}
\usepackage{colonequals}
\usepackage{graphicx} 
\usepackage{hyperref}
\usepackage{amsmath}
\usepackage{amsthm}
\usepackage{amssymb}
\usepackage{enumitem}
\usepackage[usenames,dvipsnames]{color}
\usepackage[T1]{fontenc}

\newcommand{\F}{\mathbb{F}}
\newcommand{\Q}{\mathbb{Q}}
\newcommand{\R}{\mathbb{R}}
\newcommand{\Z}{\mathbb{Z}}

\newcommand{\dual}[1]{#1^{\sharp}}

\newcommand{\pmaximal}{\rho}
\newcommand{\locmod}{\Gamma}

\numberwithin{equation}{section}

\newtheorem{cor}[equation]{Corollary}
\newtheorem{lem}[equation]{Lemma}
\newtheorem{thm}[equation]{Theorem}

\theoremstyle{definition}
\newtheorem{defn}[equation]{Definition}

\theoremstyle{remark}
\newtheorem{rem}[equation]{Remark}

\DeclareMathOperator{\GL}{GL}
\DeclareMathOperator{\Hom}{Hom}

\DeclareMathOperator{\Res}{Res}

\newcommand{\defi}[1]{{\bf #1}}

\title{Efficient enumeration of quadratic lattices}
\author{Eran Assaf, Victor Chen, Rohan Garg, and Benny Wang}
\date{\today}

\begin{document}

\begin{abstract}
    We present an algorithm to enumerate isometry classes of integral quadratic lattices of a given rank and determinant, and analyze its running time by giving bounds on the number of genus symbols for a fixed rank and determinant. We build on previous work of Kirschmer, Brandhorst, Hanke, and Dubey–Holenstein. We analyze the running times of their respective algorithms and compare the practical performance of their implementations with our own. Our implementations are publicly available.
\end{abstract}

\maketitle

\section{Introduction} \label{sec:introduction}

\subsection{Motivation}
The study of quadratic lattices has been central in mathematics for centuries, with their classification playing a major role. This traces back to Gauss \cite{Gauss}, who classified binary quadratic lattices up to isometry and introduced the concept of a genus, with important applications. Gauss's methods were used as recently as \cite{ConwaySloane} to produce tables of binary quadratic forms.

Later developments by Minkowski, Hasse, and Witt led to Eichler's complete classification of indefinite forms \cite{Eichler}, and these advances went hand in hand with tabulation efforts such as \cite{Jones35} and \cite{BrandtIntrau} for lattices of rank $3$. This continued in the work of Nipp \cite{Nipp}, who used computers to produce tables for ranks $4$ and $5$, now available in Nebe and Sloane's catalog \cite{NebeSloane}. These efforts continue today, with works such as \cite{KirschmerLorch} and \cite{EarnestHaensch} classifying all isometry classes with class number one in ranks $3$ and $4$, respectively.

Our work continues these efforts to tabulate isometry classes of lattices. We present an algorithm that constructs lattices representing each isometry class of a given rank and determinant.

\subsection{Results}

We follow the ideas of \cite{ConwaySloane}*{Chapter 15} to generate all valid genus symbols of a given determinant and rank, and use these to produce lattices in the corresponding genera. This is done by combining unpublished work of Simon Brandhorst, based on \cite{Kirschmer}*{Algorithm 3.5.6} implemented in \cite{SageMath}, with improvements coming from \cite{Hanke}. 
We compare the running time with the algorithm presented in \cite{DubeyHolenstein}.

\begin{thm} \label{thm:main theorem}
There exists an algorithm that, given positive integers $n$ and $D$ such that $\log D \ll n$, outputs an integral lattice in every genus of rank $n$ and determinant $D$, with running time
$
     n^4 D^{\alpha + \frac{\pi (1 + o(1)) }{\log \log D}}
$, where $\alpha = \log_2(3 + \sqrt{17}) - 1 \approx 1.833$.
\end{thm}

Moreover, by refining the analysis, one obtains some saving in the running-time complexity when running the algorithm for all determinants up to $D$.

\begin{thm} \label{thm:main theorem-upto}
There exists an algorithm that, given positive integers $n$ and $D$ such that $\log D \ll n$, outputs a Gram matrix of an integral lattice in every genus of rank $n$ and determinant at most $D$, with running time complexity
\[
O(n^4 D^{\alpha +1} (\log D)^{e^{\pi} - 1} \log \log D).
\]
\end{thm}

Applying Eichler's theorem of strong approximation as in \cite{ConwaySloane}*{Corollary 22}, it has the following immediate corollary.

\begin{cor} \label{cor:main theorem-indefinite}
There exists an algorithm that, given positive integers $n$ and $D$ such that $\log D \ll n$, outputs a Gram matrix of an integral lattice in every indefinite isometry class of rank $n$ and determinant at most $D$, with running time complexity
\[
O(n^4 D^{\alpha +1} (\log D)^{e^{\pi} - 1} \log \log D).
\]
\end{cor}

\begin{rem}
    For definite lattices, the running time complexity is dominated by exhausting the isometry classes within a genus, for which we do not have an improvement over the known methods of enumerating $p$-neighbors, which are exponential in the rank $n$.  
    The basic idea was first described in \cite{SchulzePillot}, and by using the isometry testing of \cite{PleskenSouvignier} instead of Minkowski reduction, it can be adapted to work for any rank, see \cite{Hanke}*{Algorithm 5.8}.
\end{rem}

An implementation of the algorithm is publicly available at \url{https://github.com/assaferan/k3s-lmfdb/tree/main/lattices}. 

A key step in the algorithm is computing a representative in a given genus. This step is the bottleneck for the running time of the overall procedure. After providing the necessary background in Section~\ref{sec:background}, in Section~\ref{sec:representative} we describe such an algorithm, originally implemented in \cite{SageMath}. We compare its running time complexity to a previously described algorithm \cite{DubeyHolenstein}.

The algorithm, implemented in \cite{SageMath} by Brandhorst, following \cite{MirMor2009}, uses two main components: the construction of a maximal integral overlattice, and local modifications to an integral lattice to match a specified $p$-adic genus symbol at a single prime $p$. These steps are described, respectively, in Sections~\ref{sec:maximal overlattices} and \ref{sec:local modifications}. 

An algorithm for finding a maximal overlattice was proposed in \cite{Hanke}. In Section~\ref{sec:maximal overlattices}, we analyze its running time and compare it with Brandhorst's. Our final algorithm incorporates these improvements into Brandhorst's implementation of Kirschmer's method.

Section~\ref{sec:complexity analysis} analyzes the running time of our algorithm and completes the proofs of Theorem~\ref{thm:main theorem} and Theorem~\ref{thm:main theorem-upto}.

\section{Background} \label{sec:background} 

In this section, we recall some needed definitions and fix notations. For reference, the interested reader may consult any standard text on quadratic forms, such as \cite{OMeara}, \cite{Cassels}, \cite{Buell}, \cite{ConwaySloane}, or \cite{Cox}.

We work in a fixed dimension $n$. Let $R$ be a PID with field of fractions $F$, and let $V$ be an $F$-vector space of dimension $n$, so that $V \simeq F^n$. 

\begin{defn}
     An \defi{$R$-lattice} in $V$ of \defi{rank} $n$ is a f.g. $R$-submodule $L \subseteq V$ such that $F L = V$. When $R=\Z$, we simply say that $L$ is a \defi{lattice}.
\end{defn}

In this paper, all lattices are equipped with a symmetric bilinear form $B$ with values in an $R$-module $N$.
The form $B$ induces a homomorphism $\phi_B : M \to \Hom_R(M,N)$, and we say that $B$ is \defi{regular} when $\phi_B$ is an isomorphism. When $R = F$ is a field, we also say that $B$ is \defi{non-degenerate}.

Given a symmetric bilinear form $B$ on $V$, any choice of a basis $E$ for $V$ gives rise to an $n \times n$ symmetric matrix $M_{B,E} \in M_n(F)$ with
\[
B(u,v) = u^\intercal M_{B,E} v \qquad \forall \, u, v \in V.
\]
This is the \defi{Gram matrix} of $B$ with respect to $E$. 

We are interested in lattices in quadratic spaces, and there is a slight difference in integrality notions between the quadratic form and the symmetric bilinear form, due to issues at $p = 2$. We therefore proceed to define quadratic structures. Let $Q$ be a quadratic form on $V$ (see \cite{Voight}*{Def. 4.2.1}), and let
\begin{equation} \label{eq:polarization}
    B(u,v) = Q(u+v)-Q(u)-Q(v)
\end{equation}
be the associated bilinear form. A \defi{quadratic space} is a pair $(V,Q)$.

\begin{defn}
    Two quadratic spaces $(V_1, Q_1)$ and $(V_2, Q_2)$ are \defi{isometric} if there exists an $F$-linear isomorphism $S : V_1 \to V_2$ such that
    $ Q_2(S(v)) = Q_1(v)$ for all $v \in V_1$.
    We then write $V_1 \simeq V_2$.
\end{defn}

\begin{defn}
    Lattices $L_1, L_2$ in quadratic spaces $V_1, V_2$ are \defi{isometric} if there exists an isometry $S : V_1 \to V_2$ with $S(L_1)=L_2$. We write $L_1 \simeq L_2$.
\end{defn}

Scaling the form, we may assume $B(L,L)\subseteq R$. This leads to the following.

\begin{defn}
    An \defi{integral $R$-lattice} in $(V,Q)$ is a lattice $L$ with $B(L,L)\subseteq R$. It is \defi{even} if $Q(L)\subseteq R$.
\end{defn}

An $R$-lattice $L$ is integral iff it has a Gram matrix $G\in M_n(R)$. As changing basis by $A\in\GL_n(R)$ replaces $G$ by $A^\intercal G A$, $\det(G)$ can only change by squares.

\begin{defn}
    Let $L$ be an integral $R$-lattice with Gram matrix $G$. Define
    \[
        \det(L)=\det(G) R^{\times 2} \in R / R^{\times 2}
    \]
    to be its \defi{determinant}. 
\end{defn}

Our goal is to classify integral lattices up to isometry for fixed rank $n$ and determinant $D$.
However, the determinant $D$ is a coarse invariant, and one obtains a finer one by working locally. For $\Z$-lattices, we can define the following invariant, measuring the local-to-global obstruction.

\begin{defn}
    Let $L_1,L_2$ be quadratic lattices. If $L_1\otimes\Z_p \simeq L_2\otimes\Z_p$ for all primes $p$ and
    $L_1\otimes\R \simeq L_2\otimes\R$,
    we say they lie in the same \defi{genus} and write $L_1\sim L_2$.
\end{defn}

For describing the genus, one can use a finite list of symbols, encoding the local isometry class at each place. At $\infty$, this is the signature. We recall below the definitions of the signature and, more generally, the local genus symbols.

\begin{defn}
    When $F=\R$, the \defi{signature} of $(V,Q)$ is the triple $(n_0,n_+,n_-)$ counting the zero, positive, and negative eigenvalues of a Gram matrix. If $n_0=0$, we write $(n_+,n_-)$.
\end{defn}

A \defi{genus symbol} is a finite collection of local invariants, one for each place $p$ (including $\infty$), that uniquely determines the genus. We use the explicit system of symbols defined in \cite{ConwaySloane}. In particular, a genus symbol is simply the tuple of all its local symbols, and two symbols are equal exactly when the associated genera are equal. Our algorithms always treat genera as given in this symbolic form, so no preprocessing such as reordering constituents is required.

In our complexity analysis, some arithmetic functions occur naturally. We denote by $\omega(D)$ the number of distinct prime factors of $D$, by $\nu_p(D)$ the valuation of $D$ at $p$, and by $\Omega(D) = \sum_{p \mid D} \nu_p(D)$ the total number of prime factors of $D$. We also denote by $\psi \approx 2.373$ the exponent for the complexity of matrix multiplication, so that multiplication of $n \times n$ matrices costs $n^{\psi}$ multiplications in the base ring.

\section{Representative of a genus} \label{sec:representative}
In this section we describe two algorithms that, given a genus symbol, return a representative lattice in that genus.
We begin by describing an algorithm, implemented in \cite{SageMath} by Brandhorst, which follows the ideas in \cite{Kirschmer}, proving its correctness, and analyzing its running time complexity. We then analyze the previous \cite{DubeyHolenstein}, and compare the two in Lemma~\ref{lem:DH vs Brandhorst}.

\subsection{Brandhorst algorithm}
\SetKwFunction{RationalRepresentative}{Rational\allowbreak-Representative}
\SetKwFunction{GramZpRepresentative}{Gram\allowbreak-$\Z_p$\allowbreak-Representative}
\SetKwFunction{GramZtwoRepresentative}{Gram\allowbreak-$\Z_2$\allowbreak-Representative}
\SetKwFunction{MaximalOverlattice}{Maximal\allowbreak-Overlattice}
\SetKwFunction{LocalModification}{Local\allowbreak-Modification}
Let \RationalRepresentative be a function that, 
given a genus symbol $\gamma$, returns a lattice in $L \otimes \Q_p$ for some $L \in \gamma$. For a $\Z_p$-genus symbol $\gamma_p$, we define \GramZpRepresentative{$\gamma_p$} 
to be the Gram matrix of a representative of $\gamma_p$.
For a lattice $L$, the function \MaximalOverlattice{$L$} returns a maximal overlattice of $L$, by using either Algorithm~\ref{alg:global_maximal} or \cite{Hanke}*{Algorithm 4.6}.
For a lattice $L$, a prime $p$, and a Gram matrix $G_p$ over $\Z_p$, the function \LocalModification($L$, $G_p$, $p$) returns a lattice $L'$ such that for all $q \ne p$, $L_q' = L_q$, and $L_p$ has Gram matrix $G_p$. This is done by applying Algorithm~\ref{alg:local modification}.
Using these building blocks, we present Algorithm~\ref{alg:sage_representative} to compute a representative lattice of a given genus.

\SetKwFunction{FRepresentative}{Representative}

\begin{algorithm}[H]
    \caption{\protect\FRepresentative{$\gamma$}; Constructs a representative of $\gamma$.}
    \label{alg:sage_representative}
    \KwData{Genus symbol $\gamma$}
    $L \gets$ \RationalRepresentative{$\gamma$}\;
    $L\gets$ \MaximalOverlattice{$2L$}\;
    $G_2\gets$ \GramZtwoRepresentative{$\gamma_2$}\;
    \eIf{$\gamma$ is not even}
    {
        $L\gets$ \LocalModification{$L$, $4 G_2$, $2$}\;
        $L\gets L/2$\;
    }
    {
        $L\gets$ \LocalModification{$L$, $G_2$, $2$}\;
    }
    \For{$2 \ne p \mid \det(\gamma)$}
    {
        $G_p\gets$ \GramZpRepresentative{$\gamma_p$}\;
        $L\gets$ \LocalModification{$L$, $G_p$, $p$}\;
    }
    \Return{$L$}\;
\end{algorithm}

\begin{thm} \label{thm:correctness of representative}
    Algorithm~\ref{alg:sage_representative} returns a lattice $L$ with $L\in\gamma$.
\end{thm}
\begin{proof}
    We show that, at the end of the algorithm, $L$ has the correct genus symbol at each prime $p$. First consider the prime $p=2$. 

    If $\gamma$ is not even, then we locally modify $L$ such that $L_2$ has Gram matrix $4G_2$. Observe that $L$ is maximal, $4G_2$ is even, and $L$ is isometric to the lattice with Gram matrix $4G_2$. Thus, by Theorem~\ref{thm:local_mod_correct}, the output of \LocalModification, call it $L'$, satisfies that $L_2'$ has a basis for which the Gram matrix is $4G_2$. Hence, $L_2'/2$ has Gram matrix $G_2$ with respect to the scaled basis.

    Consider an odd prime $p$ such that $p\mid D$. After the local modification at $p$, it is clear that $L_p$ has Gram matrix $G_p$, so $L$ indeed has the correct genus symbol at $p$.

    Finally, for odd primes $p$ that do not divide $D$, the initial choice of $L$ already had the correct genus symbols at such $p$, and by Theorem~\ref{thm:local_mod_correct} we have not modified these genus symbols. Thus, at the end of the algorithm, $L$ has the correct genus symbol at every prime $p$.
\end{proof}

Our initial complexity analysis depends on the costs of \MaximalOverlattice and \LocalModification, which vary according to the algorithms used. 

\begin{thm}
    Suppose the time complexity of doing a local modification at a prime $p$ of a lattice $L$ with rank $n$ and determinant $D$ is $\locmod(n, D, p)$. Likewise, suppose the time complexity of finding the maximal overlattice of a lattice $L$ with rank $n$ and determinant $D$ is $\pmaximal(n, D)$. Then Algorithm~\ref{alg:sage_representative} has time complexity
    \[
        O\left(\pmaximal(n, D)+\sum_{p\mid D} \locmod(n, D, p) + n^2 + n \omega(D) \right).
    \]
\end{thm}
\begin{proof}
    \RationalRepresentative{$\gamma$} takes $O(n\omega(D) + n^2)$ time, since determining the diagonal blocks of the rational Gram matrix is $O(n\omega(D))$ and assembling the full Gram matrix is $O(n^2)$.
    Similarly, each \GramZpRepresentative{$\gamma_p$} 
    can be computed in $O(n)$ time, so the total time to compute the rational representative and the local Gram matrices is
    $O(n\omega(D) + n^2)$, hence the stated complexity.
\end{proof}

As a corollary, we obtain the running time complexity when Algorithm~\ref{alg:global_maximal} is used for \MaximalOverlattice, as follows.

\begin{cor}
    Let $\delta\in(0, 1)$. For a prime $p$, let
    \begin{equation}
        c_p = \begin{cases}
            \lambda_p \log p \log \left( \frac{\omega(D)}{\delta} \right) & p \ne 2, \\
            2^{\lambda_2} & p = 2,
        \end{cases}
    \end{equation}
    where
    \[
        \lambda_p=\begin{cases}
           \min(\nu_2(D), 4n) & \text{if }p=2,\\[4pt]
           \min(\nu_p(D), n) & \text{otherwise.}
       \end{cases}
    \]
    Assume arithmetic operations take $O(1)$. Then
    when using Algorithm~\ref{alg:global_maximal} for \MaximalOverlattice,
    Algorithm~\ref{alg:sage_representative} has time complexity
    \[
        O\left(n \sum_{p \mid D} \log \nu_p(D) +
        \lambda_p n^{\psi} + c_p \right) 
    \]
    with probability of failure at most $\delta$.
\end{cor}
\begin{proof}
    The bound follows from Theorem~\ref{thm:global_maximal_time} and Corollary~\ref{cor:sage_locmod_time}, which give the stated costs for computing the maximal overlattice and for the local modifications respectively; combining those bounds with the preceding paragraph yields the claimed overall complexity and failure probability.
\end{proof}

\begin{cor} \label{cor:rep with hanke}
    When using \cite{Hanke}*{Algorithm 4.6} for \MaximalOverlattice, Algorithm~\ref{alg:sage_representative} has time complexity
    \[
    O\left(n^3 \omega(D) + n^2 \Omega(D) + n^\psi \sum_{p\mid D}\log \nu_p(D)  + n\omega(D)\Omega(D) \log D\right).
    \]
    In terms of $n$ and $\log D$, this gives the following complexity.
    \[O\left(n^3 \frac{\log D}{\log \log D}+n^\psi \log D+n\frac{(\log D)^3}{\log \log D}\right).\]
\end{cor}
\begin{proof}
    This follows from Theorem~\ref{thm:hanke_complexity} and Corollary~\ref{cor:hanke_locmod_time}. The latter time complexity is obtained by noting that $\omega(D) = O \left( \frac{\log D}{\log \log D} \right)$, $\Omega(D) = O (\log D)$, and $\sum_{p \mid 
    D} \log \nu_p(D) = O(\log D)$.
\end{proof}

\subsection{Dubey--Holenstein Algorithm} In \cite{DubeyHolenstein}, 
an algorithm to find a representative in polynomial time has been described, but no exact upper bound was determined. In this section, we summarize the algorithm and estimate its complexity in terms of the rank $n$ of and determinant $D$ of the input genus.

In \cite{DubeyHolenstein} the term quadratic form is used interchangeably with either the lattice or its Gram matrix by fixing a basis for the lattice, so we will adopt this convention throughout this section.
We use the notation $\nu_p(n)$ to denote the largest integral power of $p$ dividing $n$, where $p$ is a prime integer and $n$ is a positive integer.

At a high level, the algorithm runs in the following steps: 
\begin{enumerate}
    \item Identify an integer $t$ that has a primitive representation (see \cite{DubeyHolenstein}*{Definition 3}) in $\gamma$.
    \item Compute a quadratic form $Q_0$ that is $\overline{(Dt^{n-1})}$-equivalent (according to \cite{DubeyHolenstein}*{Definition 4} and using the notation \cite{DubeyHolenstein}*{(5)}) to $\gamma$, with its top-left entry equal to $t$.
    \item Based off of $Q_0$, generate a genus $\gamma_{n-1}$ of rank $n-1$. Recursively call the algorithm on $\gamma_{n-1}$, obtaining a representative form $Q_{n-1}$ and construct a representative based off of $Q_0$ and $Q_{n-1}$.
\end{enumerate}
The details of the algorithm can be found in \cite{DubeyHolenstein}*{Theorem 33}, and a proof of correctness can be found in \cite{DubeyHolenstein}*{Theorem 19}. The overview above is provided to help give context for our complexity analysis of the algorithm.

In our analysis, we will ignore bit complexity. In particular, we will assume that the arithmetic operations in $\Z$ and $\Z / p^k \Z$ take $O(1)$ time.
We first establish the time complexity of certain precomputations one can make prior to the main recursion of \cite{DubeyHolenstein}.

\begin{lem} \label{lem:nonQuadraticResidue}
    Let $p$ be an odd prime. There exists an algorithm which takes $p$ and a target failure probability $\delta\in(0,1)$ as input and returns an integer $k$ with $( \tfrac{k}{p}) = -1$ in
    \[O\left( \log \left( \tfrac{1}{\delta} \right) \log p \right)\]
    time, with chance of failure at most $\delta$.
\end{lem}

\begin{proof}
    Exactly half of the elements of $\F_p^\times$ are quadratic residues. Sampling uniformly from $\{1,\dots,p-1\}$ gives a non-residue with probability $1/2$; each Legendre-symbol test costs $O(\log p)$. Repeating $O(\log(1/\delta))$ independent samples yields the claimed time bound and failure probability.
\end{proof}

\begin{lem} \label{lem:threeSigns}
    Let $p \geq 5$ be an odd prime, let $t$ be an integer and let $(\varepsilon_1, \varepsilon_2, \varepsilon_3) \in \{-1,1\}^3$. There exists a randomized algorithm which takes $p,\varepsilon_1,\varepsilon_2,\varepsilon_3$ and a target failure probability $\delta\in(0,1)$ as input and returns integers $a_1,a_2,a_3$ such that
    \[a_1 + a_2 \equiv a_3 \pmod{p}, \qquad \left( \frac{a_i}{p} \right) = \varepsilon_i \text{ for } i \in \{1, 2, 3\},\]
    in $O(\log ( \tfrac{1}{\delta} ) \log p)$ time, with chance of failure at most $\delta$.
\end{lem}

\begin{proof}
    Sample $a_1,a_2$ uniformly at random from $\F_p^\times$ and set $a_3=a_1+a_2$. For a fixed triple of desired Legendre symbols $(\varepsilon_1,\varepsilon_2,\varepsilon_3)$, the proportion of pairs $(a_1,a_2)$ giving those signs is a positive constant (bounded below uniformly in $p$). Therefore $O(\log(1/\delta))$ repetitions suffice, and each repetition costs $O(\log p)$ to check the three Legendre symbols, yielding the stated bound.
\end{proof}

From here on, we will assume that the algorithms in Lemma~\ref{lem:nonQuadraticResidue} and Lemma~\ref{lem:threeSigns} have been precomputed for every single possible set of inputs for $p \le B$ for an appropriate bound $B$, so that the result of each algorithm is instantly accessible. In Theorem~\ref{thm:dubeyHolensteinAlgorithmTime}, we will determine a finite set of precomputations that are sufficient for the rest of the algorithm.
In \cite{DubeyHolenstein2}, the algorithms run with a constant probability of failure. However, given the precomputations in Lemma~ \ref{lem:nonQuadraticResidue} and Lemma~\ref{lem:threeSigns}, the failing situations are eliminated.

Let $\mathcal{G}(n,D)$ be the set of genera of rank $n$ and determinant $D$.
For notational convenience, let $e_p = \nu_p (D)$. Let $f(n, D)$ denote the maximal time complexity to compute a representative of $\gamma \in \mathcal{G}(n,D)$. 

\begin{lem} \label{lem: recursiveStepRank4}
    Suppose $n \geq 4$. We have
    \[f(n, D) = O \left( ne_2^3 + \sum_{p \mid 2D} \left( n^{1+\psi} \log e_p + n^2\omega(D)  (\log p)^2  \right)\right) + f(n-1, 2^n D)\]
\end{lem}

\begin{proof}
    We break the complexities down step by step:
    \begin{enumerate}
        \let\fullwidthdisplay\relax
        \item Our first step is to find an integer $t$ with a primitive representation in $\gamma$, such that $t \mid D$. Using \cite{DubeyHolenstein}*{Lemma 26}, this step takes
        \[O\left( \sum_{p \mid 2D } \left(\log p + \log e_p \right) \right)\]
        time. Define $q = \overline{Dt^{n-1}}$, and let $K_p = \nu_p (q)$ for each prime $p$. Note that by our construction of $t$ with $t \mid D$, the set of primes dividing $q$ is precisely $P$.
        Note that
        \[K_p = (n-1) \nu_p (t) + e_p \leq n e_p,\]
        which will be used in step (3).

        \item Next, for each prime $p \mid 2D$, we construct an $n \times n$ Gram matrix $S_p$ that is $p^{K_p}$-equivalent to $\gamma$. By construction \cite{DubeyHolenstein}*{Lemma 17}, this $S_p$ is block diagonal, with blocks of size at most $2 \times 2$. Overall, the complexity of this step is 
        $O(n\omega(D))$.
        
        \item Next, for each prime $p \mid 2D$, we construct a $1 \times n$ column vector $\mathbf x_p$ of integers representing $t$ over $\Z/ p^{K_p} \Z $, as well as an $(n-1) \times n$ matrix $A_p$ such that $[\mathbf x_p,A_p] \in \GL_n (\Z/ p^{K_p} \Z)$.
        We follow the construction outlined in \cite{DubeyHolenstein}*{Theorem 33}. By this construction, constructing $A_p$ takes $O(1)$ time, so it really suffices to compute $\mathbf x_p$. This can be done by \cite{DubeyHolenstein}*{Lemma 25}, which costs $O(\log p + \log k)$, \cite{DubeyHolenstein1}*{Lemma 17}, which costs $O(k)$, and \cite{DubeyHolenstein}*{Lemma 24}, which runs in $O(1)$. We must apply the algorithm of  \cite{DubeyHolenstein}*{Lemma 25} once for each odd prime $p \mid D$, for a total of $\omega(2D)-1$ times. This step thus runs in
        $$
        O \left( n e_2 + \sum_{p \mid D} \left( \log p + \log (n e_p) \right) \right)
        $$
        time.
        Note that, by construction, $[\mathbf x_p,A_p]$ is block diagonal with blocks of size at most $4 \times 4$.

        \item Next, for each prime $p \mid 2D$, we compute the products
        \[\mathbf d_p = \mathbf x_p^{\intercal} S_p A_p \pmod{p^{K_p}}, \qquad H_p = \big(tA_p^{\intercal} S_p A_p - \mathbf d_p^\intercal \mathbf d_p\big).\]
        The multiplication $\mathbf d_p^\intercal \mathbf d_p$ takes $O(n^2)$ time. Besides that, since $A_p$ and $S_p$ are block diagonal with blocks of size at most $4$, many of the required multiplications take $O(n)$ time; aggregating the cost over $p \mid 2D$ gives
        $ O(n^2\omega(D))$
        time.
        
        \item Next, we define $\gamma_{n-1}$ as the genus symbol with
        $\det(\gamma_{n-1}) = Dt^{n-2}$,
        and local symbol given by $H_p$ for each prime $p$ (note that $H_p$ is an $(n-1) \times (n-1)$ matrix). Define
        \[\gamma'_{n-1} = \frac{\gamma_{n-1}}{\gcd( \gamma_{n-1})}.\]
        This computation takes $O \left(\sum_{p \mid 2D} \log p \right)$. We compute a representative of $\gamma_{n-1}$ by recursively applying our algorithm on $\gamma'_{n-1}$ and scaling the result by $\gcd( \gamma_{n-1})$. It is shown in \cite{DubeyHolenstein}*{Theorem 33} that $\det(\gamma'_{n-1}) \mid \det(\gamma) \cdot 2^n$. Therefore, computing a representative of $\gamma_{n-1}$ takes at most
        \[f(n-1, 2^n D) + O \left(\sum_{p \mid 2D} \log p \right)\]
        time. 
        In particular, recall that we constructed $t$ in step (1) to be a divisor of $d$, so the set of primes dividing $Dt^{n-2}$ is the same as those dividing $2D$.
        Denote the representative of $\gamma_{n-1}$ as $\widetilde{H}$.
        
        \item From our $\omega(2D)$ vectors of the form $\mathbf x_p$, we compute a vector $\mathbf x$ of integers such that $\mathbf x \equiv \mathbf x_p \pmod{p^{K_p}}$ for each $p \mid 2D$. We analogously define $A$, $S$ and $H$. Computing these vectors/matrices is done by the Chinese Remainder Theorem on individual entries. Since $H$ is not diagonal, we have $O(n)$ nonzero entries; for each entry, there is a system of $\omega(2D)$ modular congruences we need to solve. Each entry takes $O(\omega(D)\sum_{p \mid 2D} (\log p)^2)$ time to solve. So, the overall complexity of this step is
        \begin{equation} \label{eq: step 6 DH}
        f_6(n,D) = O \left( n^2 \omega(D) \sum_{p \mid 2D} (\log p)^2 \right).
        \end{equation}

        \item For each $p \mid 2D$, we compute a matrix $U_p$ such that \[U_p^\intercal H U_p \equiv \widetilde{H} \pmod{p^{L_p}},\]
        where $L_p = \nu_p (\overline{Dt})$. 
        Note that, since $t \mid D$, we always have $L_p \leq 2e_p+ 1$ for odd $p$ and $L_p \leq 2 e_p + 3$ for $p=2$.
        Using  \cite{DubeyHolenstein2}*{Theorems 12, 14}, we complete this step in
        \begin{equation} \label{eq: step 7 DH}
        f_7(n,D) = O\left( ne_2^3 + \sum_{p \mid 2D} \left(n^{1 + \psi} \log \left( e_p \right) + n \log p \right) \right)
        \end{equation}
        time.

        \item Similarly to step (6), we compute a matrix $U$ of integers such that
        $U \equiv U_p \pmod{p^{K_p}}$
        for each $p \mid 2D$. This takes
        \begin{equation} \label{eq: step 8 DH}
        f_8(n,D) = O \left(n^2 \omega(D) \sum_{p \mid 2D} (\log p)^2 \right)
        \end{equation}
        time to solve.

        \item Finally, computing \cite{DubeyHolenstein}*{(61)} costs $ O(n^2)$.
    \end{enumerate}
    Summing these steps, we see that the dominant terms are coming from steps (5), (6), (7) and (8). Therefore
    \begin{equation} \label{eq: recursion detailed}
        f(n,D) = (1+o(1)) \sum_{i=6}^8 f_i(n,D) + f(n-1, 2^n D),
    \end{equation}
    yielding the result. 
\end{proof}

\begin{thm} \label{thm:dubeyHolensteinAlgorithmTime}
    The algorithm laid out in \cite{DubeyHolenstein}*{Theorem 33} runs in
    \[
         O\Bigg( n^2 (e_2 + n^2)^3 + n^{2+\psi} \sum_{p \mid 2D} \log e_p + n^3 \omega(D) \sum_{p \mid 2D} (\log p)^2 + \log \left( \tfrac{\omega(D)}{\delta } \right) \sum_{p \mid 2D}  \log p \Bigg).\]
    time, with at most a $\delta$ chance of failure.
\end{thm}

\begin{proof}
    In Lemma~\ref{lem: recursiveStepRank4}, we make use of the fact that we have done the precomputations in Lemma~\ref{lem:nonQuadraticResidue} and Lemma~\ref{lem:threeSigns}. As such, before the recursion, our first step is to apply these two algorithms on every prime $p \mid 2D$. To ensure that our algorithm fails with probability at most $\delta$, it suffices to allow a $\tfrac{ \delta}{2 \omega(2D) }$ chance of failure for each call of either algorithm. Thus, our precomputations take a total time of 
    \[
    O \left( \log \left( \tfrac{\omega(D)}{\delta } \right) \sum_{p \mid 2D}  \log p \right).
    \]
    
    We turn to compute the total runtime of the recursive part of the algorithm, using the recursion we computed in Lemma~\ref{lem: recursiveStepRank4}, and a similar analysis of the cases $n \le 3$. 
    From \eqref{eq: recursion detailed}, we obtain
    $$
    f(n,D) = (1+o(1)) \sum_{j=6}^8 \sum_{i=1}^n f_j \left(i, 2^{\frac{(n+i+1)(n-i)}{2}} D \right).
    $$
    By \eqref{eq: step 6 DH} and \eqref{eq: step 8 DH}, we have
    $$
    \sum_{i=1}^n f_j \left(i, 2^{\frac{(n+i+1)(n-i)}{2}} D \right) 
    = O \left( n^3 \omega(D) \sum_{p \mid 2D} (\log p)^2 \right)
    $$
    for $j = 6,8$. From \eqref{eq: step 7 DH} we get
    $$
    \sum_{i=1}^n f_7 \left(i, 2^{\frac{(n+i+1)(n-i)}{2}} D \right)
    = O \left( n^2(e_2 + n^2)^3 + n^{2 + \psi} \sum_{p \mid 2D} \log e_p + n^2 \sum_{p \mid 2D} \log p \right).
    $$
    
    Summing everything together gives us the final claimed complexity. The bottleneck with respect to $n$ lies in the $n^8$ term, which can be traced back to computing the canonical $2$-adic form of $H$ in step 7 of Lemma~\ref{lem: recursiveStepRank4}.
\end{proof}

\begin{rem}
    We note that the above bound is sharp, as for genera whose local $2$-adic symbol consists of $n$ blocks which are $1$-dimensional of type II, the recursive step increases the $2$-adic valuation of $D$ by $n$.
\end{rem}

We obtain the following immediate corollary.

\begin{cor}\label{cor:complexity-dh}
    Let $\gamma$ be a genus. Let $n$ and $D$ be the rank and determinant of $\gamma$. As just a function of $n$ and $\ell = \log D$, the algorithm laid out in \cite{DubeyHolenstein} runs in
    \[O\left(n^2 ( n^2 + \ell)^3 + n^{2+\psi} \ell \log( \ell) + n^3 \ell^2 \log( \ell)^2  + \ell^2 \log \left( \frac{\ell}{\delta} \right)\right)\]
    time, with chance of failure at most $\delta$.
\end{cor}

We justify using Algorithm~\ref{alg:sage_representative} with \cite{Hanke}*{Algorithm 4.6} over  \cite{DubeyHolenstein}.

\begin{lem} \label{lem:DH vs Brandhorst}
    For $\log D \ll n$, the running time complexity of
    Algorithm~\ref{alg:sage_representative}, when using \cite{Hanke}*{Algorithm 4.6}, is better than that of \cite{DubeyHolenstein}.
\end{lem}

 \begin{proof}
     According to Theorem~\ref{thm:dubeyHolensteinAlgorithmTime},
     the time complexity of \cite{DubeyHolenstein}
     is 
     \[
     O \left(n^8 + n^{2+\psi} \sum_{p \mid D} \log \nu_p(D) + n^3 \omega(D) \sum_{p \mid D} (\log p)^2 \right).
     \]

     According to Corollary~\ref{cor:rep with hanke}, the time complexity of Algorithm~\ref{alg:sage_representative}, when using \cite{Hanke}*{Algorithm 4.6} for the maximal overlattice algorithm, is 
     \[
     O\left(n^3\omega(D) + n^2 \Omega(D) + n^\psi \sum_{p \mid D} \log \nu_p(D) + n \omega(D) \Omega(D) \log D \right).
     \]  

     Since $\omega(D) \ll \log D \ll n$, we have $n^3 \omega(D) \ll n^8$. Clearly, 
     $$
     n^{\psi} \sum_{p \mid D} \log \nu_p(D) \ll n^{2+\psi} \sum_{p \mid D} \log \nu_p(D),
     $$
     and as $\Omega(D) \ll \log D$, we also have
     $$
     n \omega(D) \Omega(D) \log D \ll n \omega(D) (\log D)^2 
     \ll n^3 \omega(D) \sum_{p \mid D} (\log p)^2.
     $$

     Adding these three inequalities, we obtain the result.
\end{proof}

\section{Maximal overlattices} \label{sec:maximal overlattices}

In this section we describe and analyze the running-time complexity of two algorithms that, given a lattice $L$, return a maximal overlattice of $L$. The first was implemented by Simon Brandhorst in \cite{SageMath}, following \cite{Kirschmer} and the second is described in \cite{Hanke}. Although \cite{Hanke} describes and proves correctness of the algorithms, we require an analysis of the time complexities of these algorithms for our purposes; this is the content of this section.

\subsection{Brandhorst maximal overlattice algorithm}

In both this section and Section~\ref{sec:local modifications}, normal form is defined as in \cite{MirMor2009}*{Def. 4.6}. Let $M(k)$ denote the cost of multiplying $k$-bit integers. By \cite{Harvey-Hoeven}, one may take $M(k)=O(k\log k)$.

There are two parts to this algorithm: finding a maximal overlattice over $\Z_p$ for odd primes $p$, and finding a maximal overlattice over $\Z_2$. The overall algorithm simply takes the maximal overlattice over $\Z_p$ for all relevant primes $p$.

For an odd prime $p$, in order to find a maximal overlattice over $\Z_p$, one begins by $p$-saturating the lattice. Recall that an integral $\Z$-lattice $L$ is \defi{$p$-saturated} if $pD_p(L)=0$, where 
$D_p(L) = D(L_p) = \dual{L_p} / L_p$ is its \defi{$p$-discriminant} group.
\SetKwFunction{OPSaturate}{OP\allowbreak-Saturate}
\SetKwFunction{OPOverlattice}{OP\allowbreak-Overlattice}

Algorithm~\ref{alg:Zp_overlattice} implements \OPOverlattice{$L, p$}, where $L$ is a $p$-saturated $\Z$-lattice and $p$ is an odd prime. It constructs a $p$-saturated $\Z$-lattice $L'$ satisfying
\begin{enumerate}
    \item $L_p\subseteq L'_p$, with strict inclusion when $L_p$ is not $\Z_p$-maximal;
    \item $L_q=L'_q$ for all primes $q\neq p$.
\end{enumerate}
If $L_p$ is already maximal, the algorithm returns $L$.

\RestyleAlgo{ruled}
\begin{algorithm}[tbp]
    \caption{\protect\OPOverlattice{$L,p$}.} \label{alg:Zp_overlattice}
    \KwData{A lattice $L$ and an odd prime $p$}
    $\mathcal D\gets D_p(L)$;
    $Q\gets$ the quadratic form on $\mathcal D$;
    $D\gets \det(Q)$\;
    $B\gets$ a basis of $\mathcal D$ with Gram matrix in normal form\;

    \uIf{$-\nu_p(D)\le 1$}{
        \Return{$L$}\;
    }
    \uElseIf{$-\nu_p(D)=2$}{
        Let $\mathbf a, \mathbf b$ be the two elements of $B$\;
        $a\gets p\cdot Q(\mathbf a),\quad b\gets p\cdot Q(\mathbf b)$\;
        \eIf{$-b/a$ is not a quadratic residue modulo $p$}
        {\Return{$L$}\;}
        {
            Let $x$ be such that $x^2\equiv -\frac ba\pmod p$\;
            $\mathbf t\gets x\mathbf a+\mathbf b$\;
        }
    }
    \uElseIf{$-\nu_p(D)\ge 3$}{
        Choose $\mathbf a$, $\mathbf b$, $\mathbf c$ to be three elements of $B$\;
        $a\gets p\cdot Q(\mathbf a), \quad b\gets p\cdot Q(\mathbf b),\quad c\gets p\cdot Q(\mathbf c)$\;
        \eIf{$-b/a$ is a quadratic residue modulo $p$}
        {
            Let $x$ be such that $x^2\equiv -\frac ba\pmod p$\;
            $\mathbf t\gets x\mathbf a+\mathbf b$
        }
        {
            Let $x$ and $y$ be such that $ax^2+by^2+c\equiv 0\pmod p$\;
            $\mathbf t=x\mathbf a+y\mathbf b+\mathbf c$
        }
    }
    $\mathbf t'\gets$ a preimage of $\mathbf t$ in $L^\vee$\;
    \Return{$L+\langle\mathbf t'\rangle$}\;
\end{algorithm}

\begin{lem} \label{lem:Zp_overlattice_proof}
    Algorithm~\ref{alg:Zp_overlattice} produces the correct output.
\end{lem}
\begin{proof}
    Let $\mathcal D$ be the $p$-discriminant group of $L$ written in normal form, and let $D=\det(\mathcal{D})$. Because $L$ is $p$-saturated, $\mathcal D$ is an $\F_p$-vector space of dimension $-\nu_p(D)$. The lattice $L$ is maximal at $p$ iff $\mathcal D$ has no nonzero isotropic vectors.

    If $-\nu_p(D)\le 1$ then $\mathcal D$ has dimension at most $1$ and therefore contains no isotropic vectors, so $L$ is already maximal.

    If $-\nu_p(D)=2$, let $\mathbf a,\mathbf b$ be the basis elements. Writing $Q(\mathbf a)=a/p$ and $Q(\mathbf b)=b/p$ with $p\nmid a,b$, one checks that $\mathcal D$ contains a nontrivial isotropic vector iff $-b/a$ is a quadratic residue modulo $p$. If it is not, $L$ is maximal; otherwise choosing $x$ with $x^2\equiv -b/a\pmod p$ and setting $\mathbf t=x\mathbf a+\mathbf b$ yields an isotropic vector whose preimage $\mathbf t'$ produces a nontrivial $p$-overlattice $L+\langle\mathbf t'\rangle$.

    The case $-\nu_p(D)\ge 3$ is analogous: either $-b/a$ is a square and we proceed as in the previous case, or one finds $x,y$ with $ax^2+by^2+c\equiv 0\pmod p$ and sets $\mathbf t=x\mathbf a+y\mathbf b+\mathbf c$. In every case $Q(\mathbf t)$ is $\Z_p$-integral, so $L+\langle\mathbf t'\rangle$ is an integral overlattice at $p$ and strictly larger than $L$ when an isotropic vector exists. Thus the algorithm produces the desired overlattice.
\end{proof}

\begin{lem} \label{lem:Zp_overlattice_time}
    Let $\delta\in(0, 1)$. Algorithm~\ref{alg:Zp_overlattice} has time complexity
    \[
      O\!\big(n^{\psi} M(\nu_p(D)\log p)+\log p \log(1/\delta)\;M(\log p)\big)
    \]
    with probability of failure at most $\delta$.
\end{lem}
\begin{proof}
    Computing $\mathcal D$ costs $O(n^{\psi}M(\nu_p(D)\log p))$.
    Computing square roots modulo $p$ costs $O(\log p)$ field operations; each field operation has bit-cost $M(\log p)$. Repeating the attempt $O(\log(1/\delta))$ times reduces the failure probability to at most $\delta$, so the cost of root-finding contributes $O(\log p \log(1/\delta)M(\log p))$. Combining terms yields the stated bound.
\end{proof}

\SetKwFunction{OPMaximal}{OP\allowbreak-maximal}

This allows us to introduce a simple algorithm \OPMaximal to compute an overlattice of $L$ which is $p$-maximal. First, $p$-saturate the lattice and then apply Algorithm~\ref{alg:Zp_overlattice} repeatedly.
The process terminates since the determinants decrease.

\begin{thm} \label{thm:Zp_maximal_time}
    Let $\delta\in(0, 1)$ and set $\lambda_p=\min(\nu_p(D), n)$. There exists an algorithm to find a $p$-maximal overlattice with running time
    \[
      O\left(\left(n^\psi+n\log \nu_p(D)  \right)M(\nu_p(D)\log p)+
      \lambda_p n^\psi M(\lambda_p \log p) + c_p M(\log p) \right),
    \]
    and with failure probability at most $\delta$, where $c_p = \lambda_p  \log p  \log(1/\delta)$.
\end{thm}

\begin{proof}
    Saturation costs $O((n^\psi+n\log \nu_p(D) )M(\nu_p(D)\log p))$. After saturation the $p$-adic determinant exponent is at most $\lambda_p$, and each application of \OPOverlattice reduces that exponent by at least $1$, so the while-loop iterates at most $\lambda_p$ times.

    By Lemma~\ref{lem:Zp_overlattice_time}, one iteration costs 
    \[
      O\!\big(n^{\psi} M(\lambda_p\log p)+ \log p \log(1/\delta)M(\log p)\big).
    \]
    Multiplying this by $\lambda_p$ and adding the saturation cost yields the result.
\end{proof}

Algorithm~\ref{alg:Z2_maximal} implements \texttt{2-maximal($L$)}, where $L$ is a $\Z$-lattice. It constructs a $\Z$-lattice $L'$ for which $L'_2$ is $\Z_2$-maximal and $L'_p=L_p$ for all odd primes $p$.

\RestyleAlgo{ruled}
\begin{algorithm}[tbp]
    \caption{\texttt{2-maximal($L$)}.}\label{alg:maximal overlattice p=2}
    \label{alg:Z2_maximal}
    \KwData{A lattice $L$}
    $B\gets$ a basis of $\mathcal D = D_2(L)$;
    $Q\gets$ quadratic form on $\mathcal D$;
    $B'\gets \emptyset$ \;
    \For{generator $g$ in $B$}
    {
        $e\gets$ the additive order of $g$\;
        $\delta\gets$ the remainder when $Q(g)e$ is divided by $2$\;
        Add $2^{\left\lfloor\frac{\nu_2(e)+1}{2}\right\rfloor+\delta} g$ to $B'$\;
    }
    $L'\gets$ the lattice with basis $B'$;
    $I\gets$ isotropic subspace of $L'$\;
    $L'\gets L'+I$;
    $\mathcal D'\gets D_2(L')$\;
    \While{$L'$ is not maximal}
    {
        \eIf{$Q(v)\neq 0$ for all $v \in \mathcal{D}'$}
        {
            \Return{$L'$}\;
        }
        {
            $v'\gets$ a preimage of $v$ with $Q(v) = 0$ in $L'^\vee$\;
            $L'\gets L'+\langle v'\rangle$;
            $\mathcal D'\gets D_2(L')$\;
        }
    }
    \Return{$L'$}\;
\end{algorithm}

\begin{thm} \label{thm:Z2_maximal_proof}
    Algorithm~\ref{alg:Z2_maximal} produces the correct output.
\end{thm}
\begin{proof}
    The lattice with basis $B'$ is an overlattice of $L$ and ensures that the exponent of the discriminant is at most $4$. Adding an isotropic subspace $I$ to $L'$ also yields an integral overlattice, and is used to accelerate the computation.

    The final while-loop searches for nonzero isotropic vectors in the current discriminant group $\mathcal D'$; when such a vector is found its preimage is adjoined to the lattice, producing a strictly larger overlattice. As in the odd prime case, this process strictly decreases the relevant 2-adic exponent and therefore must terminate at a $\Z_2$-maximal lattice.
\end{proof}

\begin{thm} \label{thm:Z2_maximal_time}
    Let $\lambda_2=\min(\nu_2(D), 4n)$. Then Algorithm~\ref{alg:Z2_maximal} has time complexity
    \[
      O\!\Big((n^\psi+n\log \nu_2(D) )M(\nu_2(D))+\lambda_2 n^\psi M(\lambda_2)+2^{\lambda_2}\Big).
    \]
\end{thm}
\begin{proof}
    Computing the initial discriminant group costs 
    \[
      O((n^\psi+n\log \nu_2(D) )M(\nu_2(D))).
    \]
    The cardinality of $\mathcal D'$ is bounded by both $2^{\nu_2(D)}$ and $2^{4n}$, so scanning all vectors of $\mathcal D'$ to find an isotropic one takes $O(2^{\lambda_2})$ time. Each iteration of the outer while-loop recalculates the $2$-discriminant group (costing $O(n^\psi M(\lambda_2))$) and there are at most $O(\lambda_2)$ iterations. Combining these costs yields the stated bound.
\end{proof}

\SetKwFunction{MaximalOverlattice}{Maximal\allowbreak-Overlattice}
\SetKwFunction{TwoMaximal}{2\allowbreak-maximal}

\RestyleAlgo{ruled}
\begin{algorithm}[tbp]
    \caption{\protect\MaximalOverlattice{$L$}; Constructs an integral maximal overlattice of a lattice $L$.}
    \label{alg:global_maximal}
    \KwData{A lattice $L$}
    $L'\gets L$\;
    \If{$2\mid \det(L)$}
    {
        $L'\gets$ \TwoMaximal{$L'$}\;
    }
    \For{$2 \ne p \mid \det(L)$}
    {
        $L'\gets$ \OPMaximal{$L', p$}\;
    }
    \Return{$L'$}\;
    
\end{algorithm}

\begin{thm} \label{thm:global_maximal_time}
    Let $\delta\in(0, 1)$. Set $\lambda_2=\min(\nu_2(D), 4n)$ and $\lambda_p=\min(\nu_p(D), n)$ for odd primes $p$. Then Algorithm~\ref{alg:global_maximal} has time complexity 
    $$
        O\left(\sum_{ p\mid D}\left(n^\psi+n\log \nu_p(D) \right)M(\nu_p(D)\log p)+
        \lambda_p n^{\psi} M(\lambda_p \log p) + c_p \right)
    $$
    with probability of failure at most $\delta$, where
    $$
    c_p = \begin{cases}
        \lambda_p \log p \log \left( \frac{\omega(D)}{\delta} \right) M(\log p) & p \ne 2, \\
        2^{\lambda_2} & p = 2.
    \end{cases}
    $$
\end{thm}
\begin{proof}
    Apply Theorems~\ref{thm:Zp_maximal_time} and~\ref{thm:Z2_maximal_time} to each prime dividing $D$. Assigning a failure probability of at most $\delta/\omega(D)$ to each odd prime and summing the costs yields the stated global bound.
\end{proof}

\begin{cor}
    Let $\delta\in(0, 1)$. Assuming that arithmetic operations take constant time, i.e., $M(x)=1$ for all $x$, then Algorithm~\ref{alg:global_maximal} has time complexity
    \[O\left(2^{4n}+n^{\psi+1}\log D + n \log D \log \log D + n \log D \log(1/\delta)\right)\]
    with probability of failure $\delta$.
\end{cor}
\begin{proof}
    We build off of the previous theorem. Since $M(x)=1$, $\lambda_2\le 4n$, and  
    $\nu_2(D) = O(\log D)$, we get at $p = 2$
    \[
    (n^\psi+n\log \nu_2(D) )M(\nu_2(D))+\lambda_2 n^\psi M(\lambda_2)+2^{\lambda_2} = O(n\log \log D +2^{4n}).
    \]
    
    Since $\lambda_p\le n$, 
    $\nu_p(D) = O(\log D)$, and $\omega(D) = O(\log D)$, we get for $p \ne 2$
    \begin{align*}
    \left(n^\psi+n\log \nu_p(D) \right)&M(\nu_p(D)\log p) +
    \lambda_p n^\psi M(\lambda_p \log p)
    + c_p \\
    &= O\left(n^\psi+n\log \log D +n(n^\psi+ \log p\log(\log D /\delta))\right) \\
    &= O\left(n^{\psi+1}+n \log p \log \log D + n \log p \log(1/\delta) \right).
    \end{align*}
    Summing over all $p \mid D$ yields the desired result.
\end{proof}

\subsection{Hanke algorithm}
In this section we describe the running-time analysis of the algorithm described in \cite{Hanke} for finding a maximal overlattice.

Throughout this section let $Q$ be a rational quadratic form on an $n$-dimensional vector space $V$ with (absolute) determinant $D$. Fix a $\Z$-basis of $V$ and express all Gram matrices with respect to that basis; thus a lattice $L\subset V$ is represented by an $n\times n$ matrix of coordinates. 
For a Gram matrix $G\in\Q^{n\times n}$ clear denominators by letting $d=\operatorname{lcm}$ of the denominators of entries of $G$ and set $A=dG\in\Z^{n\times n}$. 
\begin{lem}\label{lem:satur}
The algorithm \cite{Hanke}*{Algorithm 3.8} takes
$
O \left( n^{\psi} \sum_{p \mid D} \log \nu_p(D) \right).
$
\end{lem}

\begin{proof}
Computing the $\operatorname{lcm}$ of $n^2$ denominators is $O(n^2)$.
Computing the dual lattice $\dual L$ reduces to inverting $A$ over $\Q$, followed by clearing denominators which can be done via Gaussian Elimination and thus takes $O(n^\psi)$ time.

The refinement loop repeatedly modifies the lattice to reduce local invariants $m_{p}$. Each iteration requires one dual-lattice computation. An initial local invariant $m_{p}$ can be reduced to 1 in $O(\log m_{p})$ steps, and summing these valuations over all primes $p \mid D$ yields $O\left(\sum_{p\mid D} \log \nu_p(D)\right)$ iterations overall. The cost per iteration is $O(n^\psi)$ and we multiply this by the number of iterations to get the result.
\end{proof}

\begin{lem} \label{lem:isotropic vector}
Finding an isotropic vector in an $n$-dimensional quadratic space over $\mathbb{F}_p$ by the algorithm \cite{Hanke}*{Algorithm 3.11} requires expected time $O(n^2\log p)$ per prime $p$. Summing over primes $p\mid D$ gives expected time $O(n^2 \Omega(D))$.
\end{lem}

\begin{proof}
The randomized reduction picks $\vec a,\vec m\in\mathbb{F}_p^n$ uniformly at random and reduces the isotropy test to solving a quadratic equation in one variable
\[
Q_B(\vec a+t\vec m)\equiv 0 \pmod p.
\]
Evaluating the quadratic polynomial and computing its discriminant takes $O(n^2)$ field operations. Testing whether the discriminant is a square and extracting a square root when it is, costs $O(\log p)$ field operations. Therefore the expected cost per prime is $O(n^2\log p)$; summing over $p \mid D$ yields the stated bound.
\end{proof}

\begin{lem} \label{lem:extend isotropic}
Extending an isotropic vector $0 \ne \vec v$ to a basis as in  \cite{Hanke}*{Algorithm 3.12} costs $O(n^2)$ per local factor, hence $O(n^2 \omega(D))$ in total.
\end{lem}

\begin{proof}
With respect to the chosen basis compute the row $r=\vec v^T G$; this multiplication costs $O(n)$. Find an index $i$ with $r_i\neq 0$ (if one exists) and perform elementary column and row operations (shears and scalings) to make $\vec v$ the first basis vector and update the remaining basis vectors accordingly. Each elementary operation updates $O(n)$ entries and at most $O(n)$ such operations are required, yielding $O(n^2)$ operations per local factor.
Summing over $p \mid D$ yields the result.
\end{proof}

\begin{lem} \label{lem:cost of isotropic}
The runnning time complexity of \cite{Hanke}*{Algorithm 3.13, 3.14} is $O(n^3)$ per local factor, hence $O(n^3 \omega(D) + n^2 \Omega(D))$ in total.
\end{lem}

\begin{proof}
Let $A=dG\in\Z^{n\times n}$ be the integral Gram matrix after clearing denominators. Computing the radical is computing the nullspace of $A$, which can be achieved by Gaussian elimination in $O(n^\psi)$ time. Once a basis for the radical is obtained, extracting one hyperbolic pair requires a sequence of elementary row/column operations (basis changes and shears) that update $O(n)$ entries per elementary operation; extracting a single hyperbolic plane therefore costs $O(n^2)$. Extracting up to $\lfloor n/2\rfloor$ planes yields at most $O(n^3)$. Adding the cost from Lemma~\ref{lem:isotropic vector} yields the result.
\end{proof}

\begin{lem} \label{lem:isotrop}
The running time complexity of \cite{Hanke}*{Algorithm 3.16} is
\[
O\left(n^\psi \sum_{p \mid D} \log \nu_p(D) + n \omega(D)\Omega(D)\log D \right).
\]
\end{lem}

\begin{proof}

For a fixed invariant factor $s_i=p^{k_1}k_2$ with $\gcd(k_2,p)=1$, lifting a $p$-primary element of the discriminant to an element of $\dual{L}$ with vanishing non-\(p\) components amounts to solving systems of modular congruences. Each such system typically has $O(\omega(D))$ modulus sizes to consider (in the sense of valuations across primes), and solving a single such modular system can be implemented in $O(\omega(D) \log D)$ time. The total number of independent lifts is at most $O(n \Omega(D))$ (each prime contributes at most $n$ directions and the sum of valuations is $O(\Omega(D))$). Multiplying these estimates gives $O(n \omega(D)\Omega(D)\log D)$ for the lifting phase. Combining this with Lemma~\ref{lem:satur} and the $O(n^\psi)$ cost for SNF, see \cite{Storjohann}, yields the stated bound.
\end{proof}

\begin{lem}[Cost of \cite{Hanke}*{Algorithm 4.6}]\label{lem:neigh}
Finding an even $2$-neighbor, or finding that none exist, can be done in $O(n^\psi)$ time. 
\end{lem}

\begin{proof}
Let $L_{\mathrm{even}}$ denote the sublattice of $L$ consisting of all vectors $v\in L$ with $B(v,v)\equiv 0\pmod 2$. Following the description in \cite{HankeNotes}, every even $2$-neighbor $L'$ of $L$ can be represented in the form
\[
\frac 12\vec w + \{ \vec v\in L : H(\vec v,\vec w)\equiv 0\pmod 2\},
\]
for some $\vec w\in L_{\mathrm{even}}^\perp$ satisfying the condition $B(\vec w,\vec w)\equiv 0\pmod 8$. 
However, if $n \ge 5$, then $\vec w$ exists and can be found in $O(1)$ by enumerating $8^5$ vectors. Building the neighbor then takes $n^\psi$. 
\end{proof}

\begin{thm} \label{thm:hanke_complexity}
Under the hypotheses above, \cite{Hanke}*{Algorithm 4.6} runs in time
\[
O\bigl(n^3\omega(D) + n^\psi \sum_{p\mid D}\log \nu_p(D) + n^2 \Omega(D) + n\omega(D)\Omega(D) \log D\bigr).
\]
\end{thm}

\begin{proof}
Sum Lemma~\ref{lem:satur}, Lemma~\ref{lem:cost of isotropic},  Lemma~\ref{lem:isotrop}, and Lemma~\ref{lem:neigh}.
\end{proof}

\section{Local modifications} \label{sec:local modifications}

\SetKwFunction{NormalForm}{NormalForm}

In this section, we describe an algorithm implemented by Brandhorst in \cite{SageMath} to locally modify a lattice, following \cite{Kirschmer}. That is, given a prime $p$, a $\Z$-lattice $M$ such that $M_p$ is maximal, and a Gram matrix $G$ of a lattice isomorphic to $M$ over $\Q_p$, we return the lattice $L$ such that $L_p$ has Gram matrix $G$ and $L_q=M_q$ over all primes $q\neq p$. Let \NormalForm{$G$} return a matrix $T$ such that $TGT^{\top}$ is is $p$-adic normal form. Then the local modification algorithm is described in Algorithm~\ref{alg:local modification}.

\RestyleAlgo{ruled}
\begin{algorithm}[H]
    \caption{\protect\LocalModification{$M, G, p$}; Locally modifying a lattice.}\label{alg:local modification}
    \KwData{A prime $p$, a $\Z_p$ maximal lattice $M$, and a Gram matrix $G$ of a lattice isomorphic to $M$ over $\Q_p$}
    $d\gets$ denominator of $G^{-1}$;
    $d \gets p^{\nu_p(d)}$\;
    
    $L\gets$ integral lattice with standard basis and Gram matrix $G$\;
    $L_{\max}\gets$ \OPMaximal{$L$};
    $G_{\max}\gets$ the Gram matrix of $L_{\max}$\;
    $L_B\gets$ the basis of $L_{\max}$;
    $T_L\gets$ \NormalForm{$G_{\max}$};
    $B\gets (T_L L_B)^{-1}$\;
    $G_M\gets$ the Gram matrix of $M$\;
    $M_B\gets$ the basis of $M$;
    $T_M\gets$ \NormalForm{$G_M$};
    $B'\gets B\cdot T_M\cdot M_B$\;
    $N\gets$ the lattice spanned by $B'$\;
    \Return{$(N\cap M)+d M$}\;
\end{algorithm}
\begin{thm} \label{thm:local_mod_correct}
    Algorithm~\ref{alg:local modification} produces the correct output.
\end{thm}
\begin{proof}
    Clearly $N\cap M + dM \subseteq M$, and for a prime $q\neq p$, $M_q = dM_q$, hence also $(N_q\cap M_q)+d M_q = M_q$. It remains to consider $\Z_p$.
    
    Let $U_L=T_L L_B$ and $U_M=T_MM_B$.
    Then $U_L$ and $U_M$ are bases for which $L_{\max}$ and $M$ respectively are in normal form, and $B'=U_L^{-1}U_M$.
    By \cite{MirMor2009}*{Prop. 4.8}, the normal form at $p$ of a $\Z_p$-maximal lattice is uniquely determined by the rational invariants. Thus, since $L_{\max}$ and $M$ are both $\Z_p$-maximal lattices, we find that $L_{\max}$ and $M$ have the same normal form $J$. Thus, if $M_I$ is the inner product matrix of $M$, 
    $$
    U_LGU_L^\top=J=U_M M_I U_M^\top \implies B'M_IB'^\top=G.
    $$
    In particular, $B'$ is the basis of the lattice in the space of $M$ that has Gram matrix $G$. Thus, $N$ is the integral lattice with Gram matrix $G$.

    To finish, it suffices to show that $(N_p \cap M_p)+dM_p = N_p$. Observe that $L \subseteq L_{\max}$, hence $U_L^{-1}$ has all entries in $\Z_p$.
    Since $B'=U_L^{-1} U_M$, it follows that $N_p \subseteq M_p$, hence $N_p \cap M_p = N_p$. 
    Since $dU_L$ has all entries in $\Z_p$, we also have $dM_p \subseteq N_p$, hence
    $$(N_p\cap M_p)+d M_p=N_p+d M_p=N_p,$$
    completing the proof.
\end{proof}
\begin{thm} \label{thm:loc_mod_time}
    Suppose the time complexity of finding the $\Z_p$-maximal overlattice of a lattice $L$ with rank $n$ and determinant $D$ is $\pmaximal(n, D, p)$. Then Algorithm~\ref{alg:local modification} has time complexity 
    $O(n^\psi+\pmaximal(n, D, p))$.
\end{thm}
\begin{proof}
    By definition, taking the $\Z_p$-maximal overlattice of $L$ has time complexity $\pmaximal(n, D, p)$. The remainder of the algorithm simply has time complexity $O(n^\psi)$, which yields our desired time complexity.
\end{proof}

\begin{cor} \label{cor:sage_locmod_time}
    Let $\delta\in(0, 1)$. Assuming that arithmetic takes $O(1)$ time, i.e., $M(x)=1$ for all $x$, then with the Algorithms \ref{alg:Zp_overlattice} and \ref{alg:Z2_maximal} for the maximal overlattice, Algorithm \ref{alg:local modification} has time complexity 
    \begin{enumerate}
        \item $O\left(n\log \nu_p(D)+\lambda_p\left(n^{\psi} + \log p \log(1/\delta)\right)\right)$ for odd primes $p$
        \item $O\left(n\log \nu_2(D) +\lambda_2 n^\psi +2^{\lambda_2}\right)$ for $p=2$
    \end{enumerate}
   where
   \[\lambda_p=\begin{cases}
       \min(\nu_2(D), 4n) & \text{if }p=2\\
       \min(\nu_p(D), n) & \text{otherwise}
   \end{cases}\]
\end{cor}
\begin{proof}
    With Algorithms \ref{alg:Zp_overlattice} and \ref{alg:Z2_maximal} for the maximal overlattice, we know from Theorems~\ref{thm:Zp_maximal_time} and~\ref{thm:Z2_maximal_time} that $\pmaximal(n, d, p)$ has time complexity
   \begin{enumerate}
        \item $O\left(n\log \nu_p(D)+\lambda_p\left(n^{\psi} + \log p \log(1/\delta)\right)\right)$ for odd primes $p$, with probability of failure at most $\delta$,
        \item $O\left(n\log \nu_2(D) +\lambda_2 n^\psi +2^{\lambda_2}\right)$ for $p=2$,
    \end{enumerate}
    where we assumed $M(x)=1$ for all $x$. Note that the remainder of the algorithm has time complexity $O(n^\psi)$, which is subsumed by the time complexity for finding the $\Z_p$-maximal overlattice of $L$, so we are done.
\end{proof}

\begin{cor} \label{cor:hanke_locmod_time}
    With \cite{Hanke}*{Algorithm 4.6}, Algorithm~\ref{alg:local modification} has time complexity
    \[O(n^3+ n^2 \log p + n^\psi\log \nu_p(D) ).\]
\end{cor}
\begin{proof}
    This follows immediately from a simple modification of Theorem~\ref{thm:hanke_complexity} for maximal overlattices over $\Z_p$.
\end{proof}

\section{Counting Genus Symbols} \label{sec:complexity analysis}

We aim to asymptotically compute the number of genus symbols with a fixed rank and determinant. In this section, we describe how to compute the number of genus symbols for an odd prime $p$ when $n>\nu_p(D)$, and we give upper and lower bounds in the case $p=2$.
Let $L_p(n,D)$ denote the number of valid $p$-adic genus symbols of rank $n$ and determinant $D$.

\begin{thm} \label{thm:odd local symbols count}
    Let $n,D\in \mathbb Z$ with $n>0$. Let $2 \ne p \mid D$ with $\nu_p(D)<n$. Then
    $$
    L_p(n,D) \sim \frac{1}{8\,\nu_p(D)}\exp\left(\pi\sqrt{\nu_p(D)}\right).
    $$
\end{thm}

\begin{proof}
Fix an odd prime $p\mid D$ and write $k=\nu_p(D)$. A $p$-adic genus symbol may be described by a choice, for each $i\ge1$, of a nonnegative integer $d_i$ equal to the multiplicity of the $p^i$-constituent and, for each nonzero constituent, of a sign $\varepsilon_{p^i}\in\{\pm1\}$ (the value of the Legendre symbol of the local determinant). The constraints on dimensions are
$$
\sum_{i\ge1} i\,d_i = k,
$$
and, if $t$ denotes the number of indices $i$ with $d_i>0$, then the vector of signs contributes a factor of $2^{\,t-1}$ (one sign is determined by the determinant condition). Thus the combinatorial count we need is $c_k$, the coefficient of $x^k$ in
$$
G(x)=\prod_{i\ge1}\bigl(1+\sum_{t\ge1}2x^{ti}\bigr)
= \prod_{i\ge1}\frac{1+x^i}{1-x^i}
= \prod_{i\ge1} (1-x^{2i-1})^{-2}(1-x^{2i})^{-1}.
$$

To obtain the asymptotic behaviour of $c_k$ we apply Meinardus's theorem (see, e.g., \cite{Meinardus}; for details see \cite{Granovsky2007}). 
The Dirichlet series associated to $G$ is
$$
L(s)= \sum_{\substack{i\ge1\\ i\text{ odd}}}\frac{2}{i^s} + \sum_{\substack{i\ge2\\ i\text{ even}}}\frac{1}{i^s}=(2-2^{-s})\zeta(s).
$$

Applying the theorem gives for $k\to\infty$ yields
$$
c_k \sim \frac{1}{8} k^{-(L(0)+1)/2}\exp\left( 2 \sqrt{k \zeta(2) \Res_{s=1} L(s)} \right) = \frac{1}{8k} \exp(\pi \sqrt{k}),
$$
hence the result.
\end{proof}

\subsection{\texorpdfstring{$p=2$}{p=2} case}

We give a crude but explicit bound for the number of $2$-adic genus symbols. The counting is less regular at $p=2$ because of the richer train\allowbreak/compartment structure described in Conway and Sloane; our estimate deliberately overcounts several choices to obtain a simple closed form.

\begin{thm} \label{thm:dyadic local symbols count}
    Let $n,D\in\mathbb Z$ with $n>0$ and $k = \nu_2(D)<n$. Then
    \begin{equation} \label{eq:number 2-adic symbols bounds}
    \frac{1}{4\sqrt{3}\,k}\exp\left(\pi\sqrt{\tfrac{2k}{3}}\right)
    \le L_2(n,D) \le 
    \frac{3}{2\sqrt{51}\,k}\left(\frac{3+\sqrt{17}}{2}\right)^{\,k}\exp\left(\pi\sqrt{\tfrac{2k}{3}}\right).
    \end{equation}
\end{thm}

\begin{proof}
Let $p(k)$ denote the number of partitions of $k$ into positive integers. By the Hardy--Ramanujan asymptotic (with a classical error term; see e.g. \cite{HR17Circle}), 
$$
p(k) \asymp \frac{1}{4k\sqrt{3}}\exp\left(\pi\sqrt{\tfrac{2k}{3}}\right)\left(1+O(k^{-1/2})\right).
$$
Hence the number of possible rank assignments is $p(k)$, and this gives the lower bound in \eqref{eq:number 2-adic symbols bounds} since at least one choice of types/oddities/signs is compatible with any fixed partition.

For the upper bound we overcount: for each part we choose a type (Type I or II), for each compartment we choose an oddity (at most 4 choices), and for each train we choose a sign (we allow a factor of $2$ even when the train might have zero total rank; this overcounts but simplifies the bound). Encode the type sequence of length $k+2$ as a binary string with $1$ indicating Type I and $0$ Type II, for valuations between $0$ and $k+1$, so that the last digit is always $0$. A compartment corresponds to an occurrence of the substring $10$; a pair of adjacent Type II blocks (substring $00$) contributes an extra factor coming from train structure. Denote by $N_b(k,x,y)$ the number of binary strings of length $k+1$ ending with $b$ with $x$ occurrences of $10$ and $y$ occurrences of $00$, and let $S_b(k) = \sum_{x,y} N_b(k,x,y) 4^x 2^y$. Then 
\begin{align*}
N_0(k,x,y) &= N_0(k-1,x,y-1) + N_1(k-1,x-1,y), \\
N_1(k,x,y) &= N_0(k-1,x,y) + N_1(k-1,x,y),
\end{align*}
showing that $S_0(k) = 2S_0(k-1) + 4S_1(k-2)$ and $S_1(k) = S_0(k-1) + S_1(k-1)$, hence
$S_0(k) = 3S_0(k-1) + 2S_0(k-2)$. Since $S_0(0) = 0$ and $S_0(1) = 6$, it follows that $S_0(k) = \frac{6}{\sqrt{17}} (\lambda_1^k - \lambda_2^k)$, where $\lambda_{1,2} = \tfrac{3 \pm \sqrt{17}}{2}$.
Multiplying this combinatorial bound by the partition count $p(k)$ (to account for choices of ranks) produces the stated upper bound.
Combining the two estimates gives \eqref{eq:number 2-adic symbols bounds}.
\end{proof}

Before we conclude the proof of our main theorems for the running time complexity estimates, the following Lemma will be useful.

\begin{lem} \label{lem: num genus symbols}
    Let us denote 
    \begin{equation} \label{eq: arith function}
    f(D) = \prod_{p \mid D} \frac{\exp\left(\pi \sqrt{\nu_p(D)}\right)}{\nu_p(D)}.
    \end{equation}
    Then for all $D$ we have
    $
    f(D) \le D^{ \frac{\pi (1 + o(1)) }{\log \log D} }.
    $
\end{lem}

\begin{proof}
    Applying Cauchy-Schwarz, we get
    $$
    \log f(D) \le \pi \sum_{p \mid D} \sqrt{\nu_p(D)}
    \le \pi \sqrt{\omega(D) \Omega(D)}.
    $$
    Since $\omega(D) \Omega(D)$ is maximized at primorials, the result follows.
\end{proof}

\subsection{Proof of Theorem~\ref{thm:main theorem}}

We are now ready to prove Theorem~\ref{thm:main theorem}. 

\begin{proof}
    We combine the results from Corollary~\ref{cor:rep with hanke}, Theorem~\ref{thm:odd local symbols count} and Theorem~\ref{thm:dyadic local symbols count}. Here, we use $e_p$ as a shorthand for $\nu_p (D)$, and let $\lambda = \tfrac{3+\sqrt{17}}{2}$. 
    Since the signature has $O(n)$ possible values, the counts of local symbols in Theorem~\ref{thm:odd local symbols count} (for odd primes) and in Theorem~\ref{thm:dyadic local symbols count} (for $p = 2$) show that there are at most
    $
    O\left(n \lambda^{e_2} f(D) \right)
    $
    genera with rank $n$ and determinant $D$, where $f(D)$ is as in \eqref{eq: arith function}. By Corollary~\ref{cor:rep with hanke}, computing a genus in each representative takes
    \[
    O\left(n^3 \frac{\log D}{\log \log D}+n^\psi \log D+n\frac{(\log D)^3}{\log \log D}\right) = 
    O\left(n^3 \frac{\log D}{\log \log D} \right)
    \]
    time, where we use the assumption that $\log D \ll n$ to see the first term is dominating.
    Multiplying everything together, the running time complexity is
     \[
     O \left( n^4 \lambda^{e_2} f(D) \frac{\log D}{\log \log D}
     \right)
     .\]
    Using Lemma~\ref{lem: num genus symbols} and $e_2 \le \log_2 D$, this is the same as
    $
    n^4 \lambda^{\log_2 D} D^{\frac{\pi (1 + o(1)) }{\log \log D}}
    $,
    which can be further simplified to
    $
    n^4 D^{\log_2 \lambda + \frac{\pi (1 + o(1)) }{\log \log D}},
    $
    proving the claim.
\end{proof}

\subsection{Proof of Theorem~\ref{thm:main theorem-upto}.}

While a naive application of Theorem 1.1 would increase the exponent of $D$ by $1$,  refining the above estimates, we give a proof of Theorem~\ref{thm:main theorem-upto}.

\begin{proof}
Fix $d \le D$.
By Corollary~\ref{cor:rep with hanke}, Algorithm~\ref{alg:sage_representative} has time complexity
\[
    O\left(n^3 \omega(d) + n^2 \Omega(d) + n^\psi \sum_{p\mid d}\log(\nu_p(d)) + n\omega(d)\Omega(d) \log d\right)
    = O(n^3 \omega(D)),
\]
where we use the hypothesis $\log D \ll n$ to see that $n^3\omega(d)$ is the dominant term.

Therefore, applying Theorem~\ref{thm:odd local symbols count} and Theorem~\ref{thm:dyadic local symbols count} again, the running time complexity for finding representatives for each genus with a given $d$ is
$$
O \left( n^4 \lambda^{e_2} f(d) \omega(d) \right) = n^4 d^{\log_2 \lambda} f(d) \omega(d).
$$
Let us denote
$$
L(z,s) =  \sum_{n=1}^{\infty} \frac{f(n) z^{\omega(n)}}{n^s}
\prod_{p} \sum_{k=0}^{\infty} \frac{f(p^k) z}{p^{ks}}
=  \prod_{p} \sum_{k=0}^{\infty} \frac{z e^{\pi \sqrt{k}}}{p^{ks}}.
$$
Writing $L(z,s) = \zeta(s)^{ze^{\pi}} G(z,s)$, we have
$$
\sum_{d=1}^{\infty} \frac{f(d) \omega(d)}{d^s} = \frac{\partial}{\partial z} L(z,s) \vert_{z = 1}
 = \zeta(s)^{e^{\pi}} \frac{\partial}{\partial z} G(z,s) \vert_{z = 1} + e^{\pi} \zeta(s)^{e^{\pi}} \log \zeta(s) G(1,s).
$$
It follows that
$$
\sum_{d=1}^{\infty} \frac{f(d) \omega(d)}{d^s} \sim e^{\pi} G(1,1) \zeta(s)^{e^{\pi}} \log \zeta(s)
\sim - e^{\pi} G(1,1) (s-1)^{-e^{\pi}} \log (s-1)
$$
as $s \to 1^+$. 
By \cite{Delange}*{Th{\'e}or{\'e}me IV}, it follows that
$$
\sum_{d=1}^{D} f(d) \omega(d) \sim e^{\pi} \frac{G(1,1)}{\Gamma(e^{\pi})} D (\log D)^{e^{\pi} - 1} \log \log D.
$$
Applying Abel summation then yields the result.
\end{proof}

\section{Acknowledgments}

The authors would like to thank Simon Brandhorst for his comments on an earlier version of this manuscript, the MIT PRIMES program for facilitating this research opportunity, and in particular Dr. Thomas R{\"u}d and Prof. Tanya Khovanova for their helpful comments.
Assaf was supported by a grant from the Simons Foundation (SFI-MPS-Infrastructure-00008651, AS).

\end{document}